\documentclass[11pt,letterpaper]{article}
\usepackage{amsfonts, amsmath, amssymb, amscd, amsthm, color, graphicx, mathrsfs, wasysym, setspace, mdwlist, calc, float, mathtools, tikz-cd}

\usepackage{tocloft}
\usepackage{xcolor}

\usepackage{hyperref}

\hypersetup{colorlinks,
    linkcolor={red!50!black},
    citecolor={blue!80!black},
    urlcolor={blue!80!black}}
\usepackage{float}

\renewcommand{\phi}{\varphi}
\renewcommand{\bar}{\overline}

\newcommand{\NN}{\mathbb{N}}
\newcommand{\ZZ}{\mathbb{Z}}

\newcommand{\G}{\mathcal G}

\newcommand{\Wr}{\,{\rm Wr}\,}

\newtheorem{thm}{Theorem}[section]
\newtheorem*{thm*}{Theorem}
\newtheorem{cor}[thm]{Corollary}
\newtheorem{lem}[thm]{Lemma}
\newtheorem{prop}[thm]{Proposition}

\theoremstyle{definition}
\newtheorem{defn}[thm]{Definition}

\theoremstyle{remark}
\newtheorem{rem}[thm]{Remark}

\newcommand\blfootnote[1]{%
  \begingroup
  \renewcommand\thefootnote{}\footnote{#1}%
  \addtocounter{footnote}{-1}%
  \endgroup
}

\newfont{\eufm}{eufm10}

\begin{document}
\title{\vspace{-8mm}Condensed groups in product varieties}

\author{D. Osin\thanks{This work has been supported by the National Science Foundation grant DMS-1612473 and the Focused Research Group grant DMS-1853989.}}
\date{}

\maketitle
\vspace*{-6mm}

\begin{abstract}
A finitely generated group $G$ is said to be \emph{condensed} if its isomorphism class in the space of finitely generated marked groups has no isolated points. We prove that every product variety $\mathcal{UV}$, where $\mathcal{U}$ (respectively, $\mathcal{V}$) is a non-abelian (respectively, a non-locally-finite) variety, contains a condensed group. In particular, there exist condensed groups of finite exponent. As an application, we obtain some results on the structure of the isomorphism and elementary equivalence relations on the set of finitely generated groups in $\mathcal{UV}$.
\end{abstract}

\section{Introduction}

Let $\G$ denote the space of finitely generated marked groups. Informally, $\G$ is the set of all pairs $(G,A)$, where $G$ is a group and $A$ is an ordered finite generating set of $G$, considered up to a natural equivalence relation and endowed with the topology induced by local convergence of Cayley graphs. Given a finitely generated group $G$, we denote by $[G]\subseteq \G $ its \emph{isomorphism class}; that is,
$$
[G]=\{ (H,B)\in \G\mid H\cong G \}.
$$

The following definition is inspired by connections between the topological properties of isomorphism classes in $\G$ and model theory.

\begin{defn}
A finitely generated group $G$ is said to be \emph{condensed} if $[G]$ has no isolated points.
\end{defn}

As  shown in \cite{Osi20}, condensed groups lead to non-trivial examples of subspaces of $\G$ satisfying a topological zero-one law for $\mathcal L_{\omega_1, \omega}$-sentences. In addition, the existence of a condensed group in a closed subset $\mathcal S\subseteq \G$ has strong consequences for the structure of the isomorphism and elementary equivalence relations on $\mathcal S$ (see Theorem \ref{01} below).

Known examples of condensed groups include finitely generated groups isomorphic to their direct square, generic torsion-free lacunary hyperbolic groups \cite{Osi20}, certain solvable groups \cite{Wil}, and the iterated monodromy group of the polynomial $z^2+i$ (see \cite{Nek}). On the other hand, condensed groups do not occur among linear groups, nor among finitely presented Hopfian groups. It is unknown whether a condensed group can be finitely presented. For more details, we refer the reader to \cite{Osi20} and references therein.

The goal of this note is to suggest an elementary construction of condensed groups based on wreath products. Namely, we prove the following.

\begin{thm}\label{main}
Let $B$ be a non-abelian group and let $H$ be a finitely generated infinite group. The unrestricted wreath product $B\Wr H$ contains a condensed subgroup.
\end{thm}

This theorem provides new examples of condensed groups satisfying non-trivial laws. Using results of \cite{Osi20} we also obtain some information about the structure of the isomorphism and elementary equivalence relations on the set of finitely generated groups in decomposable varieties.

\begin{cor}\label{cor1}
Suppose that $\mathcal{U}$ (respectively, $\mathcal{V}$) is a non-abelian (respectively, non-locally-finite) variety of groups. Then the product variety $\mathcal{UV}$ contains a condensed group. In particular, the following hold:
\begin{enumerate}
\item[(a)] the isomorphism relation on the subspace $\{ (G,A)\in \G \mid G\in \mathcal {UV}\}$ is non-smooth;
\item[(b)] $\mathcal {UV}$ contains a subset of cardinality $2^{\aleph_0}$ consisting of pairwise non-isomorphic, elementarily equivalent, finitely generated groups.
\end{enumerate}
\end{cor}

An equivalence relation $E$ on a topological space $X$ is called \emph{smooth} if there is a Polish space $P$ and a Borel map $\beta \colon X\to P$ such that for any $x,y\in X$, we have $xEy$ if and only if $\beta(x)=\beta(y)$. Informally, claim (a) of Corollary \ref{cor1} means that finitely generated groups in $\mathcal{UV}$ cannot be ``explicitly classified" up to isomorphism using invariants from a Polish space. For more on complexity of Borel equivalence relations in $\G$, see \cite{TS}.

Recall also that two groups are \emph{elementarily equivalent} if they satisfy the same first order sentences in the language of groups. Claim (b) of Corollary \ref{cor1} shows that elementary equivalence in $\mathcal{UV}$ is much weaker than isomorphism. It is worth noting that finding examples of elementarily equivalent, non-isomorphic, finitely generated groups is a rather non-trivial task since the standard tools for constructing models -- ultrapowers, omitting types, and the L\" owenheim-Skolem theorem -- are not available in this case.

It was previously known that the variety of solvable groups of class $3$ contains condensed groups. Indeed, this fact easily follows from a result of Williams \cite{Wil} as explained in \cite[Example 2.8]{Osi20}. Corollary \ref{cor1} implies that such examples can already be found among (nilpotent of class $2$)-by-abelian groups. In contrast, abelian-by-nilpotent groups cannot be condensed \cite[Proposition 6.2]{Osi20}. It would be interesting to know whether condensed groups exist in some other solvable varieties, e.g., in the variety of center-by-metabelian groups.

Yet another variety to which Corollary \ref{cor1} applies is the Burnside variety $\mathcal B_n$ of all groups of exponent $n=n_1n_2$, where $n_1>2$ and $n_2$ is any number for which $\mathcal B_{n_2}$ is not locally finite (e.g., we can take $n=1995$ \cite{A}). In particular, we obtain the following.

\begin{cor}
There exist condensed groups of finite exponent.
\end{cor}

In the next section, we collect some preliminary information about the space of finitely generated groups, wreath products, and group varieties. The proofs of the main results are given in Section 3.

\section{Preliminaries}

\paragraph{2.1. The space of finitely generated marked groups.} We begin with the definition suggested by Grigorchuck \cite{Gri}. Let $\G_n$ denote the set of equivalence classes of all pairs $(G,A)$, where $G$ is a group and $A$ an ordered generating set of $G$ of cardinality $n$,  modulo the following equivalence relation: $$(G, (a_1, \ldots, a_n))\approx (H, (b_1, \ldots , b_n))$$ if the map $a_1\mapsto b_1,\; \ldots,\; a_n\mapsto b_n$ extends to an isomorphism $G\to H$. To simplify our notation, we write $(G,A)$ for the $\approx$-class of $(G,A)$.

We also write $(G,(a_1, \ldots, a_n))\approx_r (H,(b_1, \ldots, b_n))$ for some $r\in \mathbb N$ if there is an isomorphism (in the category of directed graphs) between the balls of radius $r$ around the identity in the Cayley graphs of $G$ and $H$ with respect to the generating sets $(a_1, \ldots, a_n)$ and $(b_1, \ldots, b_n)$, respectively, that takes edges labeled by $a_i$ to edges labeled by $b_i$ for all $i$. The topology on $\G_n$ is defined by taking the sets
\begin{equation}\label{UGA}
U_{G,A}(r)=\{ (H,B)\in \G_n\mid (H,B)\approx _r (G,A)\},
\end{equation}
where $(G,A)$ ranges in $\G$ and $r\in \NN$, as the base of neighborhoods. Thus, a sequence $\{(G_i, A_i)\}_{i\in \NN}$ converges to $(G,A)$ in $\G_n$ if for every $r\in \mathbb N$, $(G_i,A_i)$ and $(G,A)$ are $r$-similar for all sufficiently large $i$. It is easy to see that $\approx$ is the intersection of the equivalence relations $\approx_r$ and hence the topology on $\G_n$ is well-defined.

We record an elementary observation, which will be used in the next section.

\begin{lem}\label{GAHB}
Let $(G,A), (H,B)\in \G_n$, where $A=(a_1, \ldots, a_n)$ and $B=(b_1, \ldots, b_n)$. Given a word $w_A$ in the alphabet $A^{\pm 1}$, let $w_B$ denote the word obtained from $w_A$ by replacing all occurrences of $a_i^{\pm 1}$ with $b_i^{\pm 1}$ for all $i$. Suppose that a word $w_A$ of length at most $2r$ in the alphabet $A^{\pm 1}$ represents $1$ in $G$ if and only if the corresponding word $w_B$ in the alphabet $B^{\pm 1}$ represents $1$ in $H$. Then $(G,A)\approx _r (H,B)$.
\end{lem}

\begin{proof}
It is easy to see that the map $w_A\mapsto w_B$ induces the required isomorphism between the balls of radius $r$ around the identity in the Cayley graphs of $(G,A)$ and $(H,B)$.
\end{proof}

Identifying $(G,(a_1, \ldots, a_n))$ with $(G,(a_1, \ldots, a_n, 1))$ gives rise to an embedding $\G_n \subseteq \G_{n+1}$. The topological union $$\G=\bigcup_{n\in \NN} \G_n$$ is called the \emph{space of marked finitely generated groups}.

It turns out that there is a group of homeomorphisms of $\G$ whose orbits are exactly the isomorphism classes. One obvious corollary of this fact is the following.

\begin{prop}[{\cite[Corollary 6.1]{Osi20}}]\label{dich}
For every finitely generated group $G$, the isomorphism class $[G]$ is either discrete or has no isolated points in $\G$.
\end{prop}

\paragraph{2.2. Wreath products and varieties of groups.} For any groups $B$ and $H$, we denote by $B^H$ the set of all functions $f\colon H\to B$. We think of $B^H$ as a group with respect to pointwise multiplication. Thus, $B^H$ is isomorphic to the direct product of $|H|$ copies of $B$.

The (unrestricted) wreath product of $B$ and $H$, denoted $B \Wr H$, is the split extension of $B^H$ by $H$ where the action of $H$ on $B^H$ by conjugation is given by the formula
\begin{equation}\label{hfh}
(hfh^{-1}) (x)= f(h^{-1}x)
\end{equation}
for all $h,x\in H$ and $f\in B^H$.

Recall that a \emph{variety} is a class of groups satisfying a fixed system of identities. Alternatively, varieties can be characterized via Birkhoff's theorem: a class of groups is a variety if and only if it is closed under taking homomorphic images, subgroups, and (unrestricted) direct products. For more details on varieties of groups and all unexplained notation, the reader may consult the book \cite{Neu}.

Given two group varieties $\mathcal{U}$ and $\mathcal V$,  the product variety $\mathcal{UV}$ is defined to be the class of all extensions of groups from $\mathcal U$ by groups from $\mathcal V$. In particular, for any $B\in \mathcal U$ and $H\in \mathcal V$, we have $B\Wr H\in \mathcal{UV}$.

\section{Condensed subgroups of wreath products}

\paragraph{3.1. Auxiliary results from topological dynamics.}
Given a group $H$, we denote by $2^H$ the space of all subsets of $H$ endowed with the product topology, which coincides with the topology of pointwise convergence of characteristic functions in our case. Thus, $2^H$ is a compact Hausdorff space. The group $H$ acts on $2^H$ by homeomorphisms via left and right multiplication. In what follows, we refer to these actions as the \emph{left} and \emph{right} actions of $H$.

An action of a group $H$ on a topological space is said to be \emph{topologically transitive} if for every non-empty open sets $U$, $V$, there exists $h\in H$ such that $h(U)\cap V\ne \emptyset$. The following result is well-known. We provide a proof for completeness.

\begin{lem}
For every infinite group $H$, both left and right actions of $H$ on $2^H$ are topologically transitive.
\end{lem}

\begin{proof}
We only prove the claim for the left action; the argument for the right action is symmetric. Let $U$, $V$ be any non-empty open subsets of $2^H$ and let $S\in U$, $T\in V$. By the definition of the topology on $2^H$, there exist  finite subsets $E,F\subset H$ such that for any $R\subseteq H$ satisfying $R\cap E=S\cap E$ (respectively, $R\cap F=T\cap F$), we have $R\in U$ (respectively, $R\in V$). Since $H$ is infinite, there is $h\in H$ such that $hE$ and $F$ are disjoint. Let
$$
R= (hE\cap hS) \cup (T \cap F).
$$
Clearly, $h^{-1}R\cap E= h^{-1} (R\cap hE) =S\cap E$ and $R\cap F=T\cap F$. Hence, $R\in hU \cap V$.
\end{proof}

We say that a subset of $2^H$ is \emph{$H$-bi-invariant} if it is invariant under both left and right actions of $H$. Let $D_L (H)$ (respectively, $D_R (H)$) denote the set of all subsets $S\subseteq H$ whose orbits with respect to the left (respectively, right) action of $H$ are dense in $2^H$.

\begin{prop}\label{dense}
For every countably infinite group $H$, the set $D(H)=D_L(H) \cap D_R(H)$ is non-empty and $H$-bi-invariant.
\end{prop}

\begin{proof}
Since $H$ is countable, there exists a countable basis of neighborhoods $\{U_i\}_{i\in \mathbb N}$ in $2^H$. For $i\in \NN$, let
$$
V_i= \left(\bigcup_{h\in H} hU_i\right) \cap \left(\bigcup_{h\in H} U_ih\right)
$$
and
$$
C=\bigcap_{i\in \NN} V_i.
$$
Since both left and right actions of $H$ on $2^H$ are topologically transitive, the sets $\bigcup_{h\in H} hU_i$ and $\bigcup_{h\in H} U_ih$ are dense in $2^H$; clearly, both of them are also open. Therefore, $C$ is non-empty by the Baire category theorem. Fix any $S\in C$. For any non-empty open $W\subseteq 2^H$, we have $U_i\subseteq W$ for some $i$. Since $S\in V_i$, there exist $h_1, h_2\in H$ such that $h_1^{-1}S$ and $Sh_2^{-1}$ belong to $U_i \subseteq W$. Thus, the orbits of $S$ with respect to the left and right actions of $H$ are dense in $2^H$. In particular, we have $C\subseteq D(H)$ and hence $D(H)\ne \emptyset$.

It is clear that for every $S\in D(H)$ and every $h\in H$, we have $hS\in D_L(H)$. Further, we note that a subset $D\subseteq 2^H$ is dense if and only if for every finite subsets $E\subseteq F\subseteq H$, there is $Q\in D$ such that $Q\cap F=E$. Consider any $S\in D(H)$, $h\in H$, and any finite $E\subseteq F\subseteq H$. Since $\{Sf\mid f\in H\}$ is dense in $2^H$, there is $g\in H$ such that $Sg\cap h^{-1}F=h^{-1}E$. The latter equation is equivalent to $hSg \cap F=E$. This shows that $hS\in D_R(H)$ and thus we have $hS\in D(H)$. Similarly, we have $Sh\in D(H)$ for every $S\in D(H)$ and $h\in H$.
\end{proof}

\begin{rem}
The use of the Baire category theorem in the proof of Proposition \ref{dense} can be avoided. In fact, this result can also be proved by an explicit but somewhat longer inductive argument. We leave this as an exercise to the interested reader.
\end{rem}

\paragraph{3.2. Proofs of the main results.}
Henceforth, let $B$ and $H$ be as in Theorem \ref{main} and let $$ W= B\Wr H. $$ We fix any $a,b \in B$ such that $ab\ne ba$. Given $S\subseteq H$, let $G_S$ denote the subgroup of $W$ generated by $H$ and the elements $\bar a, \bar b_S\in B^H$ defined as follows:
\begin{equation}\label{ab}
\bar a (x)=\left\{ \begin{array}{cc}
                     a, & {\rm if}\; x=1, \\
                     1, & {\rm if}\; x\ne 1,
                   \end{array}
\right. \;\;\; {\rm and}  \;\;\;
\bar b_S (x)=\left\{ \begin{array}{cc}
                     b, & {\rm if}\; x\in S, \\
                     1, & {\rm if}\; x\notin S.
                   \end{array}
\right.
\end{equation}

Let $$X=\{ x_1, \ldots, x_n\}$$ be a fixed finite generating set of $H$. We denote by $\xi\colon 2^H\to \G$ the map defined by the formula
$$
\xi (S)=(G_S, Y_S),
$$
where
$$
Y_S=(x_1, \ldots, x_n, \bar a, \bar b_S).
$$

\begin{lem}\label{inj}
The map $\xi$ is injective.
\end{lem}

\begin{proof}
Suppose that $S,T\subseteq H$ and $S\ne T$. Without loss of generality, there is $s\in S\setminus T$. Using (\ref{hfh}), we obtain  $(s^{-1}\bar b_S s)(1)= \bar b_S (s)=b$ and hence $s^{-1}\bar b_S s \bar a\ne \bar as^{-1}\bar b_S s$. On the other hand, we have $(s^{-1}\bar b_T s)(1)= \bar b_T (s)=1$ and therefore $s^{-1}\bar b_T s \bar a= \bar as^{-1}\bar b_T s$. Thus, $(G_S, Y_S) \not\approx  (G_T, Y_T)$.
\end{proof}

In general, the map $\xi$ is not continuous. However, its restriction to $D_R(H)$ is. The main step in proving this fact is the following lemma. We denote by $|h|_X$ the length of an element $h\in H$ with respect to the generating set $X$ and let $$Ball_H(n)=\{ h\in H \mid |h|_X\le n\}.$$ For a word $w$ in some alphabet, we denote by $\| w\|$ its length.

\begin{lem} \label{ws}
Let $S$ and $T$ be elements of $D_R(H)$ such that
\begin{equation}\label{int}
S\cap Ball_H(2r)=T\cap Ball_H(2r)
\end{equation}
for some $r\in \NN$. A word $w_S$ of length $\| w_S\| \le r$ in the alphabet $Y_S^{\pm 1}$ represents $1$ in $W$ if and only if the word $w_T$ in the alphabet $Y_T^{\pm 1}$ obtained from $w_S$ by replacing each occurrence of $\bar b_S^{\pm 1}$ with $\bar b_T^{\pm 1}$ represents $1$ in $W$.
\end{lem}

\begin{proof}
Throughout this proof, we write $u=v$ for two words $u$, $v$ in some alphabet $A$ if they represent the same element of the free groups with the basis $A$. If $u$ and $v$ are words in the alphabet $Y_S^{\pm 1}$ or $Y_T^{\pm 1}$, we write $u=_Wv$ if $u$ and $v$ represent the same element of the group $W=B\Wr H$.

We will prove the forward implication. That is, we assume that $w_S=_W1$. In the free group with the basis $Y_S$, we can rewrite the word $w_S$ as  $w_S=w_S^\prime t,$ where $t$ is a word in the alphabet $X^{\pm 1}$ and
$$
w_S^{\prime}= u_1 \bar a^{\alpha _1}u_1^{-1} \cdot v_1\bar b_S^{\beta_1} v_1^{-1} \; \cdots\;   u_\ell \bar a^{\alpha _\ell}u_\ell^{-1} \cdot v_\ell\bar b_S^{\beta_\ell} v_\ell^{-1}
$$
for some $\alpha_i, \beta_i \in \ZZ$ and words $u_i$, $v_i$ in the alphabet $X^{\pm 1}$ of length $\| u_i\|, \| v_i\| \le r$ for all $i$. Similarly, we have $w_T=w_T^\prime t,$ where
$$
w_T^{\prime}= u_1 \bar a^{\alpha _1}u_1^{-1}  \cdot v_1\bar b_T^{\beta_1} v_1^{-1}\; \cdots\;  u_\ell \bar a^{\alpha _\ell}u_\ell^{-1} \cdot v_\ell\bar b_T^{\beta_\ell} v_\ell^{-1}
$$
in the free group with the basis $Y_T$. Since $w_S=_W1$, the word $t$ must represent $1$ in $H$. Therefore,
\begin{equation}\label{w=1}
w^\prime_S=_{W}1
\end{equation}
and to prove the lemma it suffices to show that $w^\prime_T=_{W}1$. To this end, we first prove an auxiliary result about the exponents $\beta_1 , \ldots , \beta _\ell$.

Let $\sim $ be the equivalence relation on the index set $I=\{ 1, \ldots, \ell\}$ defined by the rule
\begin{center}
$i\sim j$ if  and  only if $v_i =_H v_j$.
\end{center}
Let $$I= J_1\sqcup \ldots\sqcup J_k$$ be the associated partition into equivalence classes. For $m\in \{ 1, \ldots, k\}$, let also $$\sigma _m =\sum\limits_{i\in J_m} \beta_i.$$ We want to show that for all $m\in \{ 1, \ldots, k\}$, we have
\begin{equation}\label{bs}
b^{\sigma_m}=1.
\end{equation}

Fix any $m\in \{ 1, \ldots, k\}$. For any non-empty open subset $O$ of $2^H$, we can find infinitely many $h\in H$ such that $Sh^{-1}\in O$ since $S\in D_R(H)$. In particular, we can find $h\notin \{ u_1, \ldots, u_\ell\}$ such that $v_{i}^{-1}\in Sh^{-1}$ if and only if $i\in J_m$. By (\ref{ab}), we have $\bar a (u_i^{-1}h)=1$ for all $i\in I$ and
$$
\bar b_S(v_i^{-1}h)=\left\{\begin{array}{cc}
                             b, & {\rm if\;} i\in J_m \\
                             1, & {\rm otherwise.}
                           \end{array}
 \right.
$$
Thinking of $w^\prime_S$ as an element of $B^H$ and using (\ref{hfh}), we obtain
$$
w_S^\prime (h) = \bar a^{\alpha_1}(u_1^{-1}h)\cdot  \bar b_S^{\beta_1}(v_1^{-1}h) \;\cdots\;  \bar a^{\alpha_\ell}(u_\ell^{-1}h) \cdot \bar b_S^{\beta_\ell}(v_\ell^{-1}h) = b^{\sigma_m}
$$
in the group $B$. Thus, (\ref{bs}) follows from (\ref{w=1}).

We are now ready to prove that $w^\prime _T=_W1$. Clearly, $w^\prime _T$ also represents an element of $B^H$. We will show that $w^\prime _T(h)=1$ for all $h\in H$. There are two cases to consider.

{\it Case 1.} First assume that $|h|_X\le r$. Then for every $i=1, \ldots , \ell$, we have $$|v_i^{-1}h|_X\le | v_i|_X +|h|_X \le \| v_i\| +|h|_X \le 2r.$$ Using (\ref{hfh}), (\ref{ab}),  and (\ref{int}), we obtain
$$
v_i\bar b_T^{\beta_i} v_i^{-1} (h)= \bar b_T^{\beta_i}(v_i^{-1}h)= \bar b_S^{\beta_i}(v_i^{-1}h) = v_i\bar b_S^{\beta_i} v_i^{-1} (h).
$$
Hence, $w_T^\prime (h)= w_S^\prime (h)=1$ in $B$.

{\it Case 2.} Let  $|h|_X> r$. For every $i=1, \ldots , \ell$, we have $u_i^{-1}h\ne 1$ as $|u_i|_X\le \| u_i\| \le r$. Therefore,
$u_i \bar a^{\alpha _i}u_i^{-1}(h)=\bar a ^{\alpha_i} (u_i^{-1}h)=1$ and we obtain
$$
w^\prime_T (h) = \bar b_T^{\beta_1}(v_1^{-1}h) \;\cdots\;   \bar b_T^{\beta_\ell}(v_\ell^{-1}h) = \prod_{m=1}^k \bar b_T ^{\sigma_m}(v_{j_m}^{-1}h),
$$
where ${j_m}$ is any representative of the equivalence class $J_m$. Since $\bar b_T (v_{j_m}^{-1}h)\in \{ 1, b\}$ by the definition of $\bar b_T$, we get $\bar b_T ^{\sigma_m}(v_{j_m}^{-1}h)=1$ for all $m$ by (\ref{bs}). Thus, $w^\prime _T(h)=1$ in this case as well.
\end{proof}

\begin{cor}\label{cont}
The restriction of $\xi$ to the subspace $D_R(H)$ of $2^H$ is continuous.
\end{cor}

\begin{proof}
Note that $\xi(S)\in \G_{n+2}$ for every $S\subseteq H$. Thus we can work with $\G_{n+2}$ instead of $\G$. Consider any $S\in D_R(H)$ and any open neighborhood $\mathcal N$ of $\xi(S)=(G_S, Y_S)$ in $\G_{n+2}$. By the definition of the topology on $\G_{n+2}$, there exists $r\in \NN$ such that $\mathcal N$ contains the set $$U_{G_S,Y_S}(r)=\{ (G,A)\in \G_{n+2}\mid (G,A)\approx _r (G_S,Y_S)\}.$$

Suppose that some $T\in D_R(H)$ satisfies the condition
\begin{equation}\label{d}
S\cap Ball_H(4r)=T\cap Ball_H(4r).
\end{equation}
Combining Lemma \ref{ws} and Lemma \ref{GAHB}, we obtain $(G_T,Y_T)\approx _r(G_S,Y_S)$. This means that $\xi(T)\in \mathcal N$. Since subsets $T\subseteq H$ satisfying (\ref{d}) form an open neighborhood of $S$ in $D_R(H)$, we are done.
\end{proof}

We are now ready to prove our main result.

\begin{proof}[Proof of Theorem \ref{main}]
By Proposition \ref{dense}, there exists $S\subseteq H$ such that $hS\in D(H)$ and $Sh\in D(H)$ for all $h\in H$. In particular, the subset
$$
H(S)=\{hS\mid h\in H\}\subseteq 2^H
$$
is non-discrete and contained in $D_R(H)$.  By Lemma \ref{inj} and Corollary \ref{cont}, the image of $H(S)$ in $\G$ is also non-discrete. In the notation introduced above, we clearly have $\bar b_{hS} (x) = \bar b_S (h^{-1}x)$ for any $h,x\in H$. Therefore, $\bar b_{hS} = h \bar b_S h^{-1}$ in $W$. It follows that $G_S = G_{hS}$ for every $h\in H$ and hence $\xi(H(S))\subseteq [G_S]$. Thus, $[G_S]$ is non-discrete in $\G$. It remains to apply Proposition \ref{dich}.
\end{proof}

We say that a subset $\mathcal S\subseteq \G$ is \emph{isomorphism-invariant} if for every $(G,A)\in \mathcal S$, we have $[G]\subseteq \mathcal S$. To derive Corollary \ref{cor1}, we need some results obtained in \cite{Osi20}. For convenience of the reader, we summarize them below.

\begin{thm}\label{01}
For any isomorphism-invariant closed subset $\mathcal S\subseteq \G$, the following hold.
\begin{enumerate}
\item[(a)] The isomorphism relation on $\mathcal S$ is smooth if and only if $\mathcal S$ contains no condensed groups.
\item[(b)] If $\mathcal S$ contains a condensed group, then $\mathcal S$ contains $2^{\aleph_0}$ elementarily equivalent, pairwise non-isomorphic groups.
\end{enumerate}
\end{thm}

For part (a), see Proposition 2.7 in \cite{Osi20}. To obtain part (b), let $G$ be a condensed group such that $(G,A)\subseteq \mathcal S$ for some $A$. By Theorem 2.2 and Proposition 2.3 of \cite{Osi20}, the closure of the isomorphism class $[G]$ in $\G$ contains $2^{\aleph_0}$ elementarily equivalent, pairwise non-isomorphic groups. Since $\mathcal S$ is closed and isomorphism-invariant, it contains the closure of $[G]$ and claim (b) follows.

\begin{proof}[Proof of Corollary \ref{cor1}]
By assumption, there exists a non-abelian group $B\in \mathcal U$ and a finitely generated infinite group $H\in \mathcal V$. Applying Theorem \ref{main}, we obtain a condensed group $G\le B\Wr H\in \mathcal{UV}$. Since varieties of groups are closed under taking subgroups, we have $G\in \mathcal{UV}$.

It is well-known that for every variety $\mathcal W$, the set $\mathcal S=\{ (G,A)\in \G \mid G\in \mathcal W\}$ is closed in $\G$. Indeed, if a sequence $(G_i,A_i)$ converges to $(G,A)$ in $\G$ then every identity that holds in each $(G_i, A_i)$ also holds in $(G,A)$ by the definition of the topology on $\G$. It is also clear that $\mathcal S$ is isomorphism-invariant. Thus, claims (a) and (b) of the corollary follow from the corresponding parts of Theorem \ref{01}.
\end{proof}

\noindent \textbf{Denis Osin: } Department of Mathematics, Vanderbilt University, Nashville 37240, U.S.A.\\
E-mail: \emph{denis.v.osin@vanderbilt.edu}

\begin{thebibliography}{99}
\bibitem[Adi]{A}
S.I. Adian, The Burnside problem and identities in groups. Ergebnisse der Mathematik und ihrer Grenzgebiete, \textbf{95}. Springer--Verlag, Berlin--New York, 1979.

\bibitem[Gri]{Gri}
R. I. Grigorchuk,
Degrees of growth of finitely generated groups
and the theory of invariant means,
\emph{Math. USSR Izv.} \textbf{25} (1985), no. 2, 259-300.

\bibitem[Nek]{Nek}
V. Nekrashevych, A minimal Cantor set in the space of 3-generated groups, \emph{Geom. Dedicata} 124 (2007), 153-190.

\bibitem[Neu]{Neu}
H. Neumann, Varieties of groups. Springer-Verlag, 1967.

\bibitem[Osi]{Osi20} D. Osin, A topological zero-one law and elementary equivalence of finitely generated groups,\emph{arXiv:2004.07479}; to appear in \emph{Ann. Pure App. Logic}.

\bibitem[TS]{TS} S. Thomas, S. Schneider, Countable Borel equivalence relations. In: Appalachian set theory 2006–2012, 25-62. London Math. Soc. Lecture Note Ser. \textbf{406}, Cambridge Univ. Press, Cambridge, 2013.

\bibitem[Wil]{Wil}
J. Williams, Isomorphism of finitely generated solvable groups is weakly universal, \emph{J. Pure Appl. Algebra} \textbf{219} (2015), no. 5, 1639-1644.

\end{thebibliography}
\end{document}